\documentclass[10pt]{article}
\usepackage{ims,latexsym}
\usepackage{graphicx}
\def\Real{\hbox{I\kern-.1667em\hbox{R}}}

\newcommand{\bfun}{\left\{\begin{array}{ll}}
\newcommand{\efun}{\end{array}\right.}
\renewcommand{\P}{{\cal P}}
\renewcommand{\L}{{\cal L}}
\newcommand{\B}{{\cal B}}
\newcommand{\C}{{\cal C}}
\newcommand{\D}{{\cal D}}

\newcommand{\F}{{\cal F}}

\newtheorem{definition}{Definition}
\newtheorem{example}{Example}
\newtheorem{theorem}{Theorem}

\newtheorem{lemma}{Lemma}
\newtheorem{proposition}{Proposition}
\newcommand{\comment}[1]{}

\begin{document}

\title{Infinite Previsions and Finitely Additive Expectations}
\author{
\name{M.J. Schervish}
\corradd{Department of Statistics, Carnegie Mellon University,
  Pittsburgh, PA  15213, USA}
\position{Professor}
\department{Department of Statistics}
\organization{Carnegie Mellon University}
\city{Pittsburgh}
\address{Pittsburgh, \PA\ 15213}
\country{\usa}
\email{mark@cmu.edu}
\name{Teddy Seidenfeld}
\position{Herbert A. Simon University Professor}
\department{Departments of Philosophy and Statistics}
\organization{Carnegie Mellon University}
\city{Pittsburgh}
\address{Pittsburgh, \PA\ 15213}
\country{\usa}
\email{teddy@stat.cmu.edu}
\name{J.B. Kadane}\email{kadane@stat.cmu.edu}
\position{L. J. Savage Professor Emeritus}
\department{Department of Statistics}
\organization{Carnegie Mellon University}
\city{Pittsburgh}
\address{Pittsburgh, \PA\ 15213}
\country{\usa}
}
\AMS[62C05]{62A01}
\maketitle

\begin{abstract}
We give an extension of de Finetti's concept of coherence to unbounded
(but real-valued) random variables that allows for gambling in the
presence of infinite previsions.  We present a finitely additive
extension of the Daniell integral to unbounded random variables that
we believe has advantages over Lebesgue-style integrals in the
finitely additive setting.  We also give a general version of the
Fundamental Theorem of Prevision to deal with conditional previsions
and unbounded random variables.
\end{abstract}

\section{Introduction}
De Finetti (1974) presented a theory of finitely additive probability
in which the concept of {\em prevision} played the roles of both
probability and expected value (or expectation). De Finetti's theory
was motivated in two different, but equivalent, manners.  First, he
developed a theory of {\em coherent} gambling in which a bookie
chooses fair prices for gambles while trying to avoid uniform sure
loss.  Second, he took a decision-theoretic approach in which an agent
chooses previsions for random variables while being subject to a loss
function.  The agent tries to avoid choosing previsions such that an
alternative choice could achieve uniformly smaller loss.

De Finetti developed his theory fairly completely for the case in
which all random variables under consideration are bounded.  He
realized that unbounded random variables introduce interesting issues
for his theory, but he did not pursue those issues very far.  Crisma, Gigante,
and Millossovich (1997) and Crisma and Gigante (2001) present one form
of an extension of de Finetti's theory to unbounded random variables.  This
paper presents a number of extensions that are more in the spirit of
de Finetti's original theory.

Section~\ref{sec:df} gives a brief summary of de Finetti's theory for
random variables with finite previsions and
the notation that will be used in the rest of the paper.
Section~\ref{sec:infprev} gives our extension of coherence to infinite
previsions and compares the extension to the existing extension of  Crisma et
al. (1997) and Crisma and Gigante (2001).  Section~\ref{sec:extend}
shows how the fair price and 
decision theoretic motivations of de Finetti's theory remain
equivalent in the extension to unbounded random variables and to more
general loss functions than de Finetti originally
envisioned. Section~\ref{sec:prevexp} gives an 
introduction to finitely additive Daniell integrals along with their relation
to finitely additive expectations and coherence.    Section~\ref{sec:ftcp}
characterizes the relationship between conditional and marginal
previsions. 
Section~\ref{sec:ftp} shows how to extend a collection of coherent
previsions to an arbitrary larger collection.

\section{De Finetti's Two Definitions of Coherence}\label{sec:df}
Coherence of previsions, as de Finetti (1974, Chapter 3) formulates it, is the
criterion that a rational decision maker tries to avoid making
decisions that lead to uniformly larger loss than alternative
available decisions.

Let  $\Omega$ be a set. The elements of $\Omega$ will be 
 called {\em states} and denoted $\omega$. Subsets of $\Omega$ are
 called {\em events}.  {\em Random variables} are
 real-valued functions with domain $\Omega$, which we denote with
 capital letters. We take the liberty of identifying events with their
 indicator random variables.  That is, if $B\subseteq\Omega$, we let
 $B(\omega)$ stand for the random variable that equals 1 for all
 $\omega\in B$ and equals 0 for all $\omega\not\in B$.
The decisions that de Finetti (1974) contemplated
 are the assignments of (conditional) previsions for random variables.
\begin{definition}\label{def:coh}
{\rm Let $X$ be a random variable and let $B$ be a nonempty event.
A {\em prevision $P(X|B)$ for $X$ given $B$} is a fair price for
buying and selling $X$ under the condition that $B$ occurs.  That is,
for all real $\alpha$, the gamble that causes the agent
  to gain $\alpha B[X-P(X)]$ is considered fair (or acceptable).  If
  $B=\Omega$, we 
can denote $P(X|\Omega)=P(X)$ and call it a {\em marginal prevision for
  $X$}. 

A collection $\{P(X_i|B_i): i\in
  I\}$ of such previsions is {\em coherent$_1$} if, for every finite
  subset $\{i_1,\ldots,i_n\}\subseteq I$ and all real
  $\alpha_1,\ldots,\alpha_n$
\begin{equation}\label{eq:coh1}
\sup_\omega\sum_{j=1}^n\alpha_jB_{i_j}(\omega)[X_{i_j}(\omega)-P(X_{i_j})]\geq0.
\end{equation} 
That is, the resulting fair gambles do not allow uniformly negative
gain (also known as uniform sure loss).
 
A collection of forecasts is {\em coherent$_2$} if,
for every rival set $\{Q(X_i|B_i):i\in I\}$ of previsions for the same random
variables, every $\{i_1,\ldots,i_n\}\subseteq I$, and all nonnegative $\alpha_1,\ldots,\alpha_n$, 
\[\inf_\omega\sum_{j=1}^n\alpha_jB_{i_j}(\omega)\left\{[X_{i_j}(\omega)-P(X_{i_j})]^2-[X_{i_j}(\omega)
-Q(X_{i_j})]^2\right\}\leq0.\]
That is, no rival set of previsions leads to uniformly smaller
squared-error loss.
}\end{definition}

 De Finetti (1974, pp.~88--89) proved that a decision maker who wishes
 to be both coherent$_1$ and coherent$_2$ must choose the same
 previsions for both purposes.
Squared-error loss is a special case of a strictly proper scoring rule.
\begin{definition}\label{def:4} {\rm A {\em scoring rule} for coherent
    previsions of a random variable $X$ is a real-valued loss function
    $g$ with two real arguments: 
  a value of the random variable and a potential prevision $q$.  Let
  $\P$ be a collection of probability distributions that give finite
  prevision to $X$.
  We say that $g$ is {\em proper for $\P$} if, for every probability $P\in\P$,
  $P[g(X,q)]$ is minimized (as a function of $q$) by $q=P(X)$.
If, in addition,  only the quantity $q=P(X)$ minimizes expected score,
then the scoring rule is {\em strictly proper}.}
\end{definition} 

Some authors reserve the qualification ``strictly proper'' for scoring
rules that are designed to elicit an entire distribution, rather than
just the mean of a distribution.  (See Gneiting, 2011a, who calls the
latter kind {\em strictly consistent}.)  For the remainder of this paper, we
follow the language of Definition~\ref{def:4}, which matches the usage in
Gneiting (2011b).

Schervish, et al. (2013) consider the class of all
scoring rules of the form
\begin{equation}\label{eq:gdef}
g(x,q)=\bfun\int_x^q(v-x)d\lambda(v)&\mbox{if $x\leq q$,}\\[10pt]
 \int_q^x(x-v)d\lambda(v)&\mbox{if $x>q$,}\efun
\end{equation} 
where $\lambda$ is a measure that is mutually absolutely continuous with
Lebesgue measure and is finite on every bounded interval.  They also show
that all such scoring rules are strictly proper for the class of
probability measures that give finite mean to $g(X,q)$ for at least
one $q$.  The form (\ref{eq:gdef}) is suggested by equation (4.3) of
Savage (1971). 

Theorem~\ref{thm:equiv} in Section~\ref{sec:extend} generalizes the
equivalence of coherence$_1$ and coherence$_2$ to the class of
strictly proper scoring rules of the form (\ref{eq:gdef}).

\section{Infinite Previsions}\label{sec:infprev}
If one wishes to consider arbitrary sets of random variables,
including unbounded random variables, then there will be examples of
random variables that cannot be assigned finite previsions or
expectations.  One cannot expect a bookie to offer to pay an infinite
amount in exchange for a finite-valued random variable, no matter how
likely it is to take large values.  In particular, (\ref{eq:coh1})
makes no sense as a criterion for coherence if previsions are allowed
to be infinite.  Instead, we interpret an infinite
prevision as an offer to accept one side or the other of a gamble, but
not both sides.  For example, $P(X)=\infty$ means that the bookie
would pay an arbitrarily large amount $c$ in exchange for receiving $X$, but
she would not accept any finite amount in exchange for paying out $X$.
We express that idea formally in Definition~\ref{def:infprev}.
\begin{definition}\label{def:infprev}
{\rm Let $X$ be a random variable, and let $B$ be a nonempty event.  To say
that $P(X|B)=\infty$ means that
every finite number can be the price to buy $X$, but no
number can be the price to sell $X$, under the condition that $B$
occurs.  Similarly, to say that $P(X|B)=-\infty$ means that
every finite number can be the price to
sell $X$, but no number can be the price to buy $X$, under the
condition that $B$ occurs. The resulting gambles $\alpha B[X-c]$ that the agent is
willing to accept, namely those with $\alpha\geq0$ when
$P(X|B)=\infty$ and those with $\alpha\leq0$ when $P(X|B)=-\infty$ are
called {\em acceptable}.  We also call each finite sum of acceptable
gambles {\em acceptable}.}
\end{definition} 
We can now extend coherence$_1$ to collections that include infinite
previsions. 
\begin{definition}\label{def:cohinf}
{\rm Let $\{P(X_i|B_i): i\in I\}$ be a collection of 
previsions.  The previsions are {\em coherent$_1$} if, 
\begin{equation}\label{eq:inf0}
\sup_\omega\sum_{j=1}^n\alpha_jB_{i_j}(\omega)[X_{i_j}(\omega)-c_j]\geq0,
\end{equation} 
for all $\{i_1,\ldots,i_n\}\subseteq I$, all real
$\alpha_1,\ldots,\alpha_n$ such that $\alpha_j\geq0$ for all $j$
with $P(X_{i_j}|B_{i_j})=\infty$ and $\alpha_j\leq0$ for all $j$ with
$P(X_{i_j}|B_{i_j})=-\infty$, and all real $c_1\ldots,c_n$ such that
$c_j=P(X_{i_j}|B_{i_j})$ for each $j$ such that $P(X_{i_j}|B_{i_j})$ is finite.
That is, no acceptable gamble leads to uniform sure loss.}
\end{definition}
A necessary condition for $P(X)=\infty$ to be coherent$_1$ is that $X$ is
unbounded above, and a necessary condition for $P(X)=-\infty$ to be
coherent$_1$ is that $X$ is unbounded below.  But there are unbounded
random variables with coherent$_1$ finite previsions.
\begin{example}\label{exa:stpete}
{\rm Let $\Omega$ be the integers, and suppose that
$P(\{\omega=k\})=2^{-k}$ for each integer $k\geq1$.  This is a
countably additive probability.
Let $X(\omega)=\sum_{k=1}^\infty k\{\omega=k\}$.  The countably
additive expectation of
$X$ is 2, which is a coherent$_1$ finite prevision for $X$.  Actually,
every number greater than or equal to 2 (including $\infty$) is a
coherent$_1$ prevision for $X$.  On the other
hand, if $Y(\omega)=\sum_{k=1}^\infty 2^k\{\omega=k\}$, then the only
coherent$_1$ prevision for $Y$ is $\infty$.  Also, let
$Z(\omega)=\sum_{k=1}^\infty-k\{\omega=-k\}$. The countably additive
expectation of $Z$ is 0, but we will set $P(Z)=-\infty$.  Finally, let
$V=Y+Z$.  We cannot assign $P(V)=\infty-\infty$, however every extended real
number is a possible coherent$_1$ prevision for $V$.  The rest of
this example shows that the stated previsions are coherent$_1$.  The
random variables to which we have assigned previsions are the
indicators of the events $\{\omega=k\}$ for all integers $k$ and the
random variables $X,Y,Z,V$.  The most general acceptable gamble is a
linear combination of finitely many of the indicators and $X-x$ (with
$x\geq2$ and nonnegative coefficient if $P(X)=\infty$) together
with $Y-y$ (with $y$ real and nonnegative coefficient), $Z-z$ (with
$z$ real and nonpositive coefficient), and $V-v$ (with $v$ real and
coefficient whose sign matches the sign of the prevision if $P(V)$ is
infinite.  That is, let $k_1,\ldots,k_n$ be  integers, and let  
\[W(\omega)=\sum_{j=1}^n\alpha_j\left[\{\omega=k_j\}-2^{-k_j}\right]+\alpha_X[X-x]
+\alpha_Y[Y-y]+\alpha_Z[Z-z]+\alpha_V[V-v],\] 
where $x\geq2$, $y$, $z$ and $v$ are finite, $\alpha_Y\geq0$,
$\alpha_Z\leq0$, $\alpha_X\geq0$ if $P(X)=\infty$, and the sign of
$\alpha_V$ matches the sign of $P(V)$ if $P(V)$ is infinite. Then all
acceptable gambles are of the form of $W$.  If $\alpha_Y>0$ or if $\alpha_Z<0$ or if 
$\alpha_X>0$ or if $\alpha_V\ne0$, then $W$ clearly takes some
nonnegative values so that the supremum is at least 0.  The only cases
not handled yet have $\alpha_X\leq0$ and
$\alpha_Y=\alpha_Z=\alpha_v=0$.  In this case, 
\[W(\omega)=\sum_{j=1}^n\alpha_j\left[\{\omega=k_j\}-2^{-k_j}\right]
+\alpha_X[X-2]-\alpha_X(c-2).\]
Since $-\alpha_X(c-2)\geq0$, $W$ is no smaller than
\[W'(\omega)=\sum_{j=1}^n\alpha_j\left[\{\omega=k_j\}-2^{-k_j}\right]+\alpha_X[X-2].\]
Since $\sum_{\omega=1}^\infty W'(\omega)2^{-\omega}=0$, $\sup_\omega
W'(\omega)\geq0$, hence $\sup_\omega W(\omega)\geq0$, and the
previsions are coherent$_1$.}
\end{example} 

An alternative definition of coherent
infinite marginal previsions was presented by Crisma et al. (1997).
We repeat their definition here, and then prove that it is equivalent
to Definition~\ref{def:cohinf} for marginal previsions. 
\begin{definition}\label{def:infpi}
{\em Let $I$ be an index set, and let $\D=\{X_i:i\in I\}$ be a collection
of random variables defined on $\Omega$.  Let $P$ be an
extended-real-valued function defined on $\D$.  We say that $P$ is
{\em extended-coherent} if 
\begin{equation}\label{eq:infpi}
\inf_{\omega\in\Omega}\sum_{j=1}^n\alpha_jX_{i_j}(\omega)\leq
\sum_{j=1}^n\alpha_jP(X_{i_j})\leq
\sup_{\omega\in\Omega}\sum_{j=1}^n\alpha_jX_{i_j}(\omega),
\end{equation} 
for every finite integer $n$, all
$i_1,\ldots,i_n\in I$, and all real $\alpha_1,\ldots,\alpha_n$ such
that all infinite terms of the form $\alpha_jP(X_{i_j})$ have the same sign.}
\end{definition} 
It is not clear how Definition~\ref{def:infpi} regulates the prices
that a bookie is willing to pay/accept for various gambles with
infinite previsions. Furthermore, Definition~\ref{def:infpi} does not
apply to conditional previsions as stated.  On the other hand,
Definitions~\ref{def:infpi} and~\ref{def:cohinf} are equivalent when
applied to marginal previsions.
\begin{lemma}\label{lem:equivecc1}
A collection $\{P(X_i):i\in I\}$ of marginal previsions is
extended-coherent if and only if it is coherent$_1$. 
\end{lemma} 
\begin{proof}
For the ``if'' direction, assume that the previsions are
coherent$_1$. Let $n$ be a finite integer, let $i_1,\ldots,i_n\in I$,
and let $\alpha_1,\ldots,\alpha_n$ be real numbers such that all infinite
values of $\alpha_jP(X_{i_j})$ have the same sign. We need to show
that (\ref{eq:infpi}) holds.  First, suppose that all of the $P(X_{i_j})$ are
finite.  By coherence$_1$, we know that
\begin{eqnarray}\label{eq:check2}
\sup_{\omega\in\Omega}\sum_{j=1}^n\alpha_j[X_{i_j}(\omega)-P(X_{i_J})]&\geq&0,\\
\label{eq:check1}\sup_{\omega\in\Omega}\sum_{j=1}^n(-\alpha_j)[X_{i_j}(\omega)-P(X_{i_J})]
&\geq&0. 
\end{eqnarray} 
Inequality (\ref{eq:check1}) implies the first inequality in
(\ref{eq:infpi}), and (\ref{eq:check2}) implies the second.
Next, suppose that there are some infinite previsions among the
$P(X_{i_j})$ and all of the corresponding $\alpha_jP(X_{i_j})$ have
the same sign.  If the common sign is negative, then 
$\sum_{j=1}^n\alpha_jP(X_{i_j})=-\infty$, the second inequality in
(\ref{eq:infpi}) is trivially satisfied, and the $\alpha_j$
corresponding to infinite previsions all have the wrong signs to be
used in acceptable gambles. It follows that for all  real
$c_1,\ldots,c_n$ such that $c_j=P(X_{i_j})$ whenever $P(X_{i_j})$ is finite,
\[\sup_{\omega\in\Omega}\sum_{j=1}^n(-\alpha_j)[X_{i_j}(\omega)-c_j]\geq0,\]
which implies the first inequality in (\ref{eq:infpi}).  If the common
sign of the $\alpha_jP(X_{i_j})$ is positive, then  
$\sum_{j=1}^n\alpha_jP(X_{i_j})=\infty$, the first inequality in
(\ref{eq:infpi}) is trivially satisfied, and the $\alpha_j$
corresponding to infinite previsions all have the correct signs to be
used in acceptable gambles. It follows that for all real
$c_1,\ldots,c_n$ such that $c_j=P(X_{i_j})$ whenever $P(X_{i_j})$ is finite,
\[\sup_{\omega\in\Omega}\sum_{j=1}^n\alpha_j[X_{i_j}(\omega)-c_j]\geq0,\]
which implies the second inequality in (\ref{eq:infpi}). 

For the ``only if'' direction, assume that the previsions are extended-coherent.
Let $n$ be a finite integer, let $i_1,\ldots,i_n\in I$,
and let $\alpha_1,\ldots,\alpha_n$ be real numbers such that all infinite
values of $\alpha_jP(X_{i_j})$ are positive.  We need to show that 
\begin{equation}\label{eq:infcoh}
\sup_{\omega\in\Omega}\sum_{j=1}^n\alpha_j[X_{i_j}(\omega)-c_j]\geq0,
\end{equation} 
for all real $c_1,\ldots,c_n$ such that
$c_j=P(X_{i_j})$ for all $i$ such that $P(X_{i_j})$ is finite.  If all of the
$P(X_{i_j})$ are finite, then (\ref{eq:infcoh}) follows from
(\ref{eq:infpi}), so assume that at least one $P(X_{i_j})$ is
infinite. It follows that $\sum_{j=1}^n\alpha_jP(X_{i_j})=\infty$,
and (\ref{eq:infpi}) implies that
$\sup_{\omega\in\Omega}\sum_{j=1}^nX_{i_j}(\omega)=\infty$. 
Since, $\sum_{j=1}^n\alpha_jc_j$ is finite, no matter what $c_j$
values are chosen, (\ref{eq:infcoh}) follows.
\end{proof}

Crisma and Gigante (2001) extend
Definition~\ref{def:infpi} to conditional previsions. Their
definition imposes conditions on coherence similar to those of
Regazzini (1987) (both for bounded and for unbounded random variables)
that are designed to regulate the extreme indeterminacy of conditional
previsions given events with zero probability.  We prefer to avoid
such restrictions on the definition of coherence for reasons
illustrated by the following example. 
\begin{example}\label{exa:ciex}
{\rm Let $\Omega$ be the positive integers, and let $P(\cdot)$ be a
coherent$_2$ prevision that assigns probability 0 to every
singleton $\{n\}$ with $n$ a positive integer.  Many such marginal
previsions exist.  Let 
\[X(\omega)=\bfun \omega&\mbox{if $1\leq\omega\leq4$,}\\ 0&\mbox{otherwise,}\efun\]
and let $B=\{1,2,3,4\}$ so that $P(B)=P(X)=0$ and $BX=X$.  The gamble
\begin{equation}\label{eq:ciex}
\alpha X=\alpha(X-0)=\alpha B(X-0) 
\end{equation} 
is acceptable by all definitions of coherence.  However, the
restrictions that Regazzini (1987), Crisma and Gigante (2001), and
others impose would say that it is incoherent to assign $P(X|B)$ a
value outside of the closed interval $[1,4]$.  That is, it is
incoherent to offer the gamble $\alpha B(X-0)$ because
$0\not\in[1,4]$.  But, (\ref{eq:ciex}) shows that $\alpha B(X-0)$ is
already being offered without even contemplating what would be a coherent
value for $P(X|B)$.  Furthermore, for every real $p$, $-\alpha pB$ is
also being offered unconditionally, so that the sum
\begin{equation}\label{eq:ciex2}
\alpha B(X-0)-\alpha pB=\alpha B(X-p)
\end{equation} 
is being offered for every real $p$.  We think it is perfectly reasonable
to assign $P(X|B)$ a value between 1 and 4 if one wishes, but we don't
believe that it should be called incoherent to do otherwise.  After
all, the gambles that would be ruled out by such a declaration of
incoherence are already being offered, as (\ref{eq:ciex2}) illustrates.
}\end{example} 

The following lemma illustrates an intuitive property of infinite
previsions.
\begin{lemma}\label{lem:nonneg}
Let $B$ be an event with $P(B)>0$, and let $X$ and $Y$ be random
variables with coherent$_1$ previsions 
$P(X|B)$ and $P(Y|B)$ respectively.  If $X(\omega)\leq Y(\omega)$ for
all $\omega\in B$, then $P(X|B)\leq P(Y|B)$.
\end{lemma} 
\begin{proof}
Suppose, to the contrary, that $P(X|B)>P(Y|B)$.  Then, in particular,
$P(X|B)>-\infty$ and $P(Y|B)<\infty$.  Coherence$_1$ implies that
\begin{equation}\label{eq:3gam}
\sup_\omega\left\{\alpha B(\omega)[X(\omega)-c_X]
+\beta B(\omega)[Y(\omega)-c_Y]+\gamma[B(\omega)-P(B)]\right\}\geq0,
\end{equation} 
where $c_X=P(X|B)$ if $P(X|B)$ is finite, $c_Y=P(Y|B)$ if $P(Y|B)$ is
finite, $\alpha\geq0$ if $P(X|B)=\infty$, and $\beta\leq0$ if
$P(Y|B)=-\infty$. In (\ref{eq:3gam}), $\alpha$, $\beta$, $c_X$ and
$c_Y$ are otherwise unconstrained real numbers.  Choose $c_X>c_Y$ if
either is unconstrained.  (If both are
constrained, then $c_X>c_Y$ by assumption.) Let $\alpha=1$,
$\beta=-1$, and $\gamma=c_X-c_Y$.  Then (\ref{eq:3gam}) becomes
\[\sup_\omega\left\{B(\omega)X(\omega)-B(\omega)Y(\omega)+(c_Y-c_X)P(B)\right\}\geq0,\]
which is a contradiction because $B(\omega)[Y(\omega)-X(\omega)]\leq0$
for all $\omega$, and $(c_Y-c_X)P(B)<0$.
\end{proof} 

\section{Extension of Coherence$_2$}\label{sec:extend} 
In this section, we extend de Finetti's second coherence criterion in
two ways: we include more general scoring rules, and we accommodate
random variables with infinite previsions.  For scoring rules of the
form (\ref{eq:gdef}) we write
\begin{equation}\label{eq:dsimp}
g(x,q)=\int_x^q(v-x)d\lambda(v),
\end{equation} 
where we use the convention that an integral whose limits are in the wrong
order equals the negative of the integral with the limits in the correct order. 
It follows easily from (\ref{eq:dsimp}) that, if $a$ and $b$
are real numbers, then
\begin{equation} \label{eq:diff}
g(x,a)-g(x,b)=\lambda((a,b))[x-r(a,b,\lambda)],
\end{equation} 
where, for all $a$ and $b$,
\begin{equation}\label{eq:ra2}
r(a,b,\lambda)=\frac{\int_a^bvd\lambda(v)}{\lambda((a,b))}.
\end{equation} 
In equations (\ref{eq:diff}) and (\ref{eq:ra2}), we used the same
convention as above about integrals with limits in the wrong order.  In
particular, $\lambda((a,b))=-\lambda((b,a))$ if $a>b$. 

Equation~\ref{eq:diff} makes it clear that, if $P(X)$ is infinite,
then $P[g(X,q)]$ is going to be infinite for all $q$.  This is why
Definition~\ref{def:4} includes the clause that $P(X)$ be finite
before requiring that $P[g(X,q)]$ be minimized at $a=P(X)$.
Nevertheless, strictly proper scoring rules can still be used to
assess coherence in the spirit of coherence$_2$.  The following is
our generalization of coherence$_2$ that applies both with general scoring
rules and with infinite previsions.

\begin{definition}\label{def:coh3}
{\rm Let ${\cal C}$ be a class of strictly proper scoring rules.  Let
$\{(X_i,B_i):i\in I\}$ be a collection of pairs each consisting of a random
variable $X_i$ and a nonempty event $B_i$ with
corresponding conditional forecasts $\{p_i:i\in I\}$.  The forecasts are {\em
  coherent$_3$} if, for every finite subset $\{i_{j}:
j=1,\ldots,n\}\subseteq I$, every set of scoring rules
$\{g_1,\ldots,g_n\}\subseteq{\cal C}$, and every set $\{q_1,\ldots,q_n\}$ of
alternative forecasts,
\[\inf_\omega\sum_{j=1}^nB_{i_j}(\omega)\left[g_j(X_{i_j}(\omega),c_j)-g_j(X_{i_j}
(\omega),q_j)\right]\leq0,\]
where $c_j=p_{i_j}$ for all $j$ such that $p_{i_j}$ is finite, and
$c_j$ is finite and between $q_j$ and $p_{i_j}$ for all $j$ such that
$p_{i_j}$ is infinite. 
That is, no finite rival set of forecasts can provide a uniformly smaller sum
of scores than the original forecasts.}
\end{definition} 

\begin{theorem}\label{thm:equiv}
Let $\C$ be a class of scoring rules of the form (\ref{eq:gdef}).
A collection of conditional previsions is coherent$_1$ if and only if
it is coherent$_3$. 
\end{theorem} 
\begin{proof}
For the ``only if'' direction, assume that the conditional previsions are
coherent$_1$.  We want to show that no 
rival set of previsions provides uniformly smaller sum of
scores than the original previsions. That is, for each $(X_1,
B_1),\ldots,(X_n,B_n)$ with conditional previsions $p_1,\ldots,p_n$ and each set
$\{g_1,\ldots,g_n\}\subseteq\C$ of scoring rules and each set
$\{q_1,\ldots,q_n\}$ of rival previsions, we must show that
\begin{equation}\label{eq:equivgoal1}
\inf_\omega\sum_{i=1}^nB_i(\omega)[g_i(X_i(\omega),c_i)-g_i(X_i(\omega),q_i)]
\leq0, 
\end{equation}
where $c_i=p_i$ for all $i$ such that $p_i$ is finite, and
$c_i$ is finite and between $q_i$ and $p_i$ for all $i$ such that
$p_i$ is infinite. From (\ref{eq:diff}), 
\begin{equation}  \label{eq:scoregam1}
\sum_{i=1}^nB_i[g_i(X_i,c_i)-g_i(X_i,q_i)]=\sum_{i=1}^n\lambda_i((c_i,q_i))B_i[X_i-
r(c_i,q_i,\lambda_i)].
\end{equation} 
Because $\lambda_i((c_i,q_i))$ and $r(c_i,q_i,\lambda_i)-c_i$ have
the same sign,
\[\sum_{i=1}^n\lambda_i((c_i,q_i))B_i[r(c_i,q_i,\lambda_i)-c_i]\geq0.\]
Hence, the right-hand side of (\ref{eq:scoregam1}) is less
than or equal to
\begin{equation}\label{eq:scoregam4}
\sum_{i=1}^n\lambda_i((c_i,q_i))B_i[X_i-c_i].
\end{equation} 
The infimum of (\ref{eq:scoregam1}) is then less than or equal to
the infimum of (\ref{eq:scoregam4}).  Also
\begin{eqnarray*}
\inf_\omega\sum_{i=1}^n\lambda_i((c_i,q_i))B_i(\omega)[X_i(\omega)-c_i]&=&
-\sup_\omega\sum_{i=1}^n-\lambda_i((c_i,q_i))B_i(\omega)[X_i(\omega)-c_i]\\
&\leq&0,
\end{eqnarray*} 
where the inequality follows from coherence$_1$ of the previsions.
Hence the rival previsions do not provide uniformly smaller sum of scores
than the original previsions. 

For the ``if'' direction, we prove the contrapositive.  Assume that
the conditional previsions are incoherent$_1$.  If one of the
previsions is the 
wrong value for a constant random variable, the result is trivial, so
assume that each random variable takes at least two distinct values.
We want to find a finite set of random variable/event pairs and
corresponding conditional previsions together with a
rival set of previsions such that the provide uniformly smaller sum of
scores than the original set. That is, we need $(X_1,B_1),\ldots,(X_n,B_n)$ with
conditional previsions $p_1,\ldots,p_n$ a rival set of previsions
$q_1,\ldots,q_n$, and a set $g_1,\ldots,g_n$ of
scoring rules from $\C$ such that
\begin{equation}\label{eq:equivgoal2}
\inf_\omega\sum_{i=1}^nB_i(\omega)[g_i(X_i(\omega),c_i)-g_i(X_i(\omega),q_i)]>0,
\end{equation}
where $c_i=p_i$ for all $i$ such that $p_i$ is finite, and $c_i$ is
finite and between $q_i$ and $p_i$ for all $i$ such that $p_i$ is
infinite.  In principle, some of the rival $q_i$ could be infinite,
but the rivals that we construct below will all be finite.

By incoherence$_1$, there exist
$(X_1,B_1),\ldots,(X_n,B_n)$ with conditional previsions $p_1,\ldots,p_n$,
$\alpha_1,\ldots,\alpha_n$ and $\epsilon>0$ such that 
\begin{equation}\label{eq:scoregam2}
\sup_\omega\sum_{i=1}^n\alpha_iB_i(\omega)[X_i(\omega)-c_i]=-\epsilon,
\end{equation} 
where $c_i$ is finite for all $i$, $c_i=p_i$ for all $i$ such that
$p_i$ is finite, and $\alpha_i$ 
has the same sign as $p_i$ for all $i$ such that $p_i$ is infinite.
Without loss of generality, we can assume that $\max_i|\alpha_i|=1$. Let
$g$ be a scoring rule in $\C$ with
corresponding measure $\lambda$.  Let $g_i=g$ for all $i$.  Define
\[z_0=\min\left\{\min_{i=1,\ldots,n}\{\lambda((-\infty,c_i)):\alpha_i>0\},
\min_{i=1,\ldots,n}\{\lambda((c_i,\infty)):\alpha_i<0\}\right\},\]
so that $z_0>0$.
For each $z\in(0,z_0)$, define $q_i(z)$ by the equation
\[\lambda((c_i,q_i(z)))=-z\alpha_i,\]
which is possible because $\lambda((c_i,q))$ is continuous in $q$. Define
\begin{equation}\label{eq:scoregam6}
\ell(z)=\inf_\omega\sum_{i=1}^nB_i(\omega)
[g(X_i(\omega),c_i)-g(X_i(\omega),q_i(z))].
\end{equation} 
By construction $c_i-r(c_i,q_i(z),\lambda)$, $c_i-q_i(z)$, and
$r(c_i,q_i(z),\lambda)-q_i(z)$ have the same sign as $\alpha_i$ for
each $i$, and $c_i-q_i(z)$ has the largest absolute value of the three
differences.  This means that, for all $i$ such that $p_i$ is infinite
and all $z$, $c_i$ is between $q_i(z)$ and $p_i$. The remainder of the  
proof consists of showing that there exists 
$z\in(0,z_0)$ such that $\ell(z)>0$.  The desired rival previsions are
then $q_1(z),\ldots,q_n(z)$.  From (\ref{eq:diff}),  
\begin{eqnarray}\label{eq:scoregam7}\nonumber
\sum_{i=1}^nB_i[g(X_i,c_i)-g(X_i,q_i(z))]&=&\sum_{i=1}^n\lambda((c_i,q_i(z)))
B_i[X_i-r(c_i,q_i(z),\lambda)]\\
&=&-\sum_{i=1}^nz\alpha_iB_i[X_i-c_i]\nonumber
-\sum_{i=1}^nz\alpha_iB_i[c_i-r(c_i,q_i(z),\lambda)]\\
&\geq&z\epsilon-z\sum_{i=1}^n\alpha_iB_i[c_i-q_i(z)].
\end{eqnarray} 
Since, for all $i$, $q_i(0)=c_i$ and $q_i(z)$ is continuous in $z$,
there exists $z_1>0$ such that, for all $z\in(0,z_1)$ and all $\omega$,
\begin{equation}\label{eq:scoregam3}
\sum_{i=1}^n\alpha_iB_i(\omega)[c_i-q_i(z)]<\frac{\epsilon}{2}.
\end{equation} 
Choose $z\in(0,\min\{z_0,z_i\})$. Then combine (\ref{eq:scoregam6}),
(\ref{eq:scoregam7}) and (\ref{eq:scoregam3}) to conclude that
\[\ell(z)\geq z\frac{\epsilon}{2}>0,\]
which completes the proof.
\end{proof}

\section{Prevision and Expectation}\label{sec:prevexp}
The expectation of a random variable $X$ defined on $\Omega$ is
usually defined as the integral of $X$ over the set $\Omega$ with
respect to the underlying probability measure defined on subsets $\Omega$.  In
the countably additive setting, such integrals can be defined (except
for certain cases involving $\infty-\infty$)
uniquely from a probability measure on $\Omega$.
Dunford and Schwartz (1958, Chapter III) give a detailed analysis of
integration, with respect to finitely additive measures, that attempts to
replicate the uniqueness of integrals.  Their analysis requires
additional assumptions if one wishes to integrate
unbounded random variables.  

An alternative to defining integrals with respect to specific measures
is to define integrals as special types of linear functionals.  Then
measures can be constructed from the integrals.  (The integral of the
indicator of a set is the measure of the set.) This is the approach
used in the study of the {\em Daniell integral}. (See Royden, 1968, Chapter
13. Regazzini, 1987 and Williams, 2007 takes a similar approach for bounded random
variables only.) De Finetti's concept of 
prevision turns out to be a finitely additive generalization of the
Daniell integral.  (See Definition~\ref{def:daniell} below.) One major
difference between the finitely 
additive Daniell integral and the theory developed by Dunford and
Schwartz is that coherence$_1$ is the only assumption needed to define
a finitely additive Daniell integral on an arbitrary space of random
variables, including unbounded random variables and linear
combinations of random variables with infinite integrals.  Another difference
is that multiple finitely additive Daniell integrals can lead to the same
finitely additive probability distribution on
$\Omega$.  Put another way, the finitely
additive Daniell integral is not uniquely determined by its
corresponding finitely additive probability.

In this paper, we take the approach of defining finitely additive
expectations and probabilities in terms of finitely additive Daniell
integrals, rather than integrals of the sort  developed by Dunford and
Schwartz (1958).  This section is devoted to deriving and
illustrating the properties of finitely additive Daniell integrals and
their relation to coherent$_1$ previsions.

\begin{definition}\label{def:daniell}{\rm 
Let $\L$ be a linear space of real-valued functions defined on
$\Omega$ that contains all constant functions, and
let $L$ be an extended-real-valued functional defined on $\L$. 
If ($X,Y\in\L$ and $X\leq Y$) implies $L(X)\leq L(Y)$, we say that $L$ is
{\em nonnegative}.  
We call $L$ an {\em extended-linear functional} on $\L$, if, for all
real $\alpha,\beta$ and all $X,Y\in\L$,
\begin{equation}\label{eq:el}
L(\alpha X+\beta Y)=\alpha L(X)+\beta L(Y),
\end{equation} 
whenever the arithmetic on the right-hand side of (\ref{eq:el}) is
well defined (i.e., not $\infty-\infty$) and where $0\times\pm\infty=0$.
A nonnegative extended-linear functional is called a {\em finitely
  additive Daniell integral}. (See Schervish et al, 2008a.)
If $L(1)=1$, we say that $L$ is {\em normalized}.  A normalized
finitely additive Daniell integral is called a {\em finitely
  additive expectation}.}
\end{definition} 

The following simple result contains the primary justification for the
names {\em finitely additive Daniell integral} and {\em finitely additive
  expectation}.
\begin{lemma}\label{lem:fap}
Let $L$ be a finitely additive Daniell integral on $\L$, where $\L$ contains
indicators of some subsets of $\Omega$.  
\begin{enumerate} 
\item The restriction of $L$
to those subsets of $\Omega$ whose indicators are in $\L$ is a finitely
additive measure.  If $L(1)=1$, then $L$ is a finitely additive
probability.\label{part:fap1} 
\item The restriction of $L$ to the simple functions in $\L$ matches
  the usual definition of integral of a simple function with respect
  to the measure in part~\ref{part:fap1}.
\end{enumerate} 
\end{lemma} 
\begin{proof} 
\begin{enumerate} 
\item Because the constant 1 is in $\L$, the 
indicator of $\Omega$ itself is in $\L$. If $A$ and $B$ are disjoint,
the indicator of $A\cup B$ is $A+B$.  If $A$, $B$ and $A\cup B$ are
all in $\L$, then $L(A)$ and $L(B)$ are both nonnegative, and
linearity give $L(A\cup B)=L(A)+L(B)$, and the sum is well defined.
So $L$ is a finitely additive measure.
If $L(1)=1$, nonnegativity implies that $L$, restricted to the set of
indicators of events, is a finitely additive probability on $\Omega$.
\item Let $X=\sum_{i=1}^n\alpha_i A_i$ be a simple function with each
  $A_i\in\L$, all $A_i$ disjoint, and all $\alpha_i$ distinct. Then
  linearity gives $L(X)=\sum_{i=1}^n\alpha_iL(A_i)$, which is well
  defined and which matches the usual definition of the integral of $X$ with
  respect to the measure $L$. 
\end{enumerate} 
\end{proof}

Finitely additive expectations are like integrals in several ways,
e.g., the two parts of Lemma~\ref{lem:fap} as well as the linearity in
Lemma~\ref{lem:missing} below.
Another way in which finitely additive expectations are like integrals
is continuity with respect to the uniform metric.  The proof of the following
result is trivial and omitted.
\begin{proposition} \label{lem:unif}
Let $X$ and $Y$ be elements of a linear space $\L$ of real-valued
functions defined on $\Omega$.  Suppose that
$\sup_\omega|X(\omega)-Y(\omega)|\leq\epsilon$. Then, for each
finitely additive expectation $L$ on $\L$, either
$|L(X)-L(Y)|\leq\epsilon$ or $L(X)$ and $L(Y)$ equal the same infinite value.
\end{proposition} 
A special application of Proposition~\ref{lem:unif} is to a bounded function
$X$.  Every bounded function can be approximated arbitrarily closely
by simple functions.  So long as all of the approximating simple
functions are in $\L$, their finitely additive expectations will be
arbitrarily close to the expectation of $X$.  For example, if $\L$ is
the set of functions that are measurable with respect to a
$\sigma$-field, then  every bounded function in $\L$ is uniformly
approximable by simple functions in $\L$.
The expectations of all bounded random variables are then uniquely
determined from the finitely additive probability on $\L$.
Hence, when the finitely additive expectation defined here is 
restricted to bounded functions measurable with respect to a
$\sigma$-field, it is the same as the definition of integral developed
by Dunford and Schwartz (1958), and it is the same as the integral
used by Dubins (1975) in his results about disintegrability.

On a linear space, coherent$_1$ previsions are the same as finitely
additive expectations.  One direction is simpler to prove than the other.
\begin{lemma}\label{lem:faecp}
Let $L$ be a finitely additive expectation on a linear space
$\L=\{X_i:i\in I\}$. Then $L$ is a coherent$_1$ marginal prevision.
\end{lemma} 
\begin{proof} 
Suppose, 
to the contrary, that a finitely additive expectation $L$ is
incoherent$_1$ when used as a marginal prevision.  From the definition
of finitely additive expectation, the domain of $L$ is a linear space
that contains all constants.  Incoherence$_1$ implies that there exist
$i_1,\ldots,i_n\in I$, real numbers $\alpha_1,\ldots,\alpha_n$, and
real numbers $c_1,\ldots,c_n$ such that $\alpha_j$ has the same
sign as $L(X_{i_j})$ when $L(X_{i_j})$ is infinite,
$c_j=L(X_{i_j})$ when $L(X_{i_j})$ is finite, and
\[\sup_\omega\sum_{j=1}^n\alpha_j[X_{i_j}(\omega)-c_j]<0.\]
It follows that, there is $\epsilon>0$ such that
\[\sup_\omega\sum_{j=1}^n\alpha_jX_{i_j}(\omega)=-\epsilon+\sum_{j=1}^n\alpha_jc_j.\]
By nonnegativity and linearity of $L$, 
\begin{equation}\label{eq:lbound}
L\left(\sum_{j=1}^n\alpha_jX_{i_j}\right)\leq-\epsilon+\sum_{j=1}^n\alpha_jc_j.
\end{equation} 
But
\begin{equation}\label{eq:psum}
L\left(\sum_{j=1}^n\alpha_jX_{i_j}\right)=\sum_{j=1}^n\alpha_jL(X_{i_j}),
\end{equation} 
because the arithmetic on the right-hand side of (\ref{eq:psum}) is well defined
(positive if infinite).  A $+\infty$ value on the right-hand side of
(\ref{eq:psum}) would contradict (\ref{eq:lbound}) as would a finite value.
\end{proof} 

The other direction requires an additional result that is
useful in its own right.
\begin{lemma}\label{lem:missing}
Let $\D=\{X_i:i\in I\}$ be a set of random variables that contains all
constants, and let $Q$ be a coherent$_1$ prevision on $\D$.  Let
$\L$ be that part of the linear span of $\D$ that consists of linear
combinations of the form $Y=\sum_{j=1}^n\alpha_jX_{i_j}$ where each
$X_{i_j}\in\D$, and $p_Y=\sum_{j=1}^n\alpha_jQ(X_{i_j})$ is well-defined.
Then $P(Y)=p_Y$ for all $Y\in\L$ is well-defined, and it is the unique
coherent$_1$ extension of $Q$ to $\L$.
\end{lemma} 
\begin{proof}
We show first that $P$ is
well-defined, and then that it is the unique coherent$_1$
extension of $Q$ to $\L$.  Throughout the proof, we rely on
Lemma~\ref{lem:equivecc1} because the proof is slightly less
cumbersome when worded in terms of extended-coherence. To see that $P$
is well-defined, suppose 
to the contrary that $Y\in\L$ has two representations
\begin{equation}\label{eq:twoside}
Y=\sum_{j=1}^n\alpha_jX_{i_j}=\sum_{k=1}^m\beta_kX_{\ell_k},
\end{equation} 
such that
\begin{equation}\label{eq:tworep}
\sum_{j=1}^n\alpha_jQ(X_{i_j})\ne\sum_{k=1}^m\beta_kQ(X_{\ell_k}),
\end{equation} 
where both sides of (\ref{eq:tworep}) are well-defined sums. That
\begin{equation}\label{eq:twodiff}
0\equiv\sum_{j=1}^n\alpha_jX_{i_j}-\sum_{k=1}^m\beta_kX_{\ell_k}
\end{equation} 
is immediate from (\ref{eq:twoside}).  
Regardless of which sides (if any) of (\ref{eq:tworep}) are finite or infinite,
\begin{equation}\label{eq:twosum}
\sum_{j=1}^n\alpha_jQ(X_{i_j})-\sum_{k=1}^m\beta_kQ(X_{\ell_k})\ne0,
\end{equation} 
and the sum on the left-hand side of (\ref{eq:twosum}) is well-defined.
This contradicts extended-coherence when combined with
(\ref{eq:twodiff}).  

That $P$ is extended-coherent is immediate from
Definition~\ref{def:infpi}.  It is also clear that $P$ extends $Q$ to
$\L$.  For each $Y\in\L$, to see that $p_Y$ is the unique
extended-coherent value for $P(Y)$, let 
$Y\in\L$ be represented as stated and let $p$ be an extended real
number such that 
\begin{equation}\label{eq:pnotsum}
p\ne\sum_{j=1}^n\alpha_jQ(X_{i_j}).
\end{equation}
We prove that setting $P(Y)=p$ is extended-incoherent.  The sum
\begin{equation} \label{eq:pandsum}
-p+\sum_{j=1}^n\alpha_jQ(X_{i_j})
\end{equation}
is well-defined, regardless of
whether or not either side or both sides of (\ref{eq:pnotsum}) are finite.
Then $0\equiv -Y+\sum_{j=1}^n\alpha_jX_{i_j}$ but (\ref{eq:pandsum}) is
not 0, which makes $P(Y)=p$ extended-incoherent when combined with
existing previsions.
\end{proof} 

We can now prove the converse to Lemma~\ref{lem:faecp}.
\begin{lemma}\label{lem:cpfae} 
Let $P$ be a coherent$_1$ marginal prevision on a linear space $\L$
that contains all constants. Then $P$ is a finitely additive expectation.
\end{lemma} 
\begin{proof}
That $P(1)=1$ is immediate from coherence$_1$.  Lemma~\ref{lem:nonneg}
tells us that $X\leq Y$ implies
$P(X)\leq P(Y)$.  That $P$ is extended-linear follows from Lemma~\ref{lem:missing}.
\end{proof}  

The ways in which coherent$_1$ previsions are more general than
finitely additive expectations are the following:
\begin{itemize}
\item The domain of a coherent$_1$ prevision need not be a linear
  space.
\item The domain of a coherent$_1$ prevision need not contain all constants.
\item The definition of finitely additive expectation does not, as it
  stands, deal with conditional expectations.
\end{itemize} 
It is trivial to extend the domain of a coherent$_1$ prevision to
include all constants.  To extend the domain of a coherent$_1$
prevision beyond what Lemma~\ref{lem:missing} provides is the subject
of the Fundamental Theorem of Prevision (Theorem~\ref{thm:ftp}).
The results in Section~\ref{sec:ftcp} give
the tools needed to give meaning to finitely additive conditional expectation.

\section{Relationship Between Conditional Previsions and Marginal
  Previsions}\label{sec:ftcp}
The relationship between marginal previsions and integrals is much
more intuitive than the relationship between general conditional
previsions and integrals.  Even in the countably additive theory,
conditional expectations are defined as Radon-Nikodym derivatives
rather than as integrals.  In the finitely additive
theory used here, we defined finitely additive expectations (marginal previsions)
as a type of integral (see Section~\ref{sec:prevexp}).  To a large
extent, conditional previsions and marginal previsions determine each other.
We make that statement precise in this section.

The results of this section fall into two categories:
\begin{itemize}
\item There are some relationships between conditional and marginal
  previsions that must hold in order for them to be jointly
  coherent$_1$.  For example
\[P(XB)=P(B)P(X|B),\]
whenever the the product on the right is not $0$ times an infinite value.
\item There are some cases with $P(B)=0$ in which $P(X|B)$ is
  completely unconstrained by other specified previsions.
\end{itemize} 

The first result is that, when $P(B)>0$, $P(X|B)$ is uniquely determined from
$P(B)$ and $P(XB)$. 
\begin{lemma}\label{lem:gt0}
Let $\D=\{(X_i,B_i):i\in I\}$ be a collection of random
variable/nonempty event pairs with coherent$_1$ conditional previsions
$\{P(X_i|B_i):i\in I\}$.  Let $(X,B)$ be a random variable/nonempty
event pair such that $(B,\Omega)$ and $(XB,\Omega)$ are in $\D$ and
$P(B)>0$.  Then the only possible coherent$_1$
value for $P(X|B)$ is $P(XB)/P(B)$ (even if the
numerator is infinite.)
\end{lemma} 
\begin{proof}
Let $P(B)=q$.  First, assume that $P(XB)=p$ is finite.
We start by showing that the gamble $\alpha B(X-p/q)$ is 
acceptable for all real $\alpha$.  We know that $\alpha(XB-p)$ and
$-(p/q)\alpha(B-q)$ are acceptable for all real $\alpha$.  Hence the
following gamble is acceptable:
\[\alpha(XB-p)-(p/q)\alpha(B-q)=\alpha B(X-p/q).\]
To see that every value other than $p/q$ is incoherent$_1$, suppose
that $P(X|B)=r\not=p/q$.  Then the following gamble would be
acceptable for all real $\alpha$:
\[\alpha(XB-p)-r\alpha(B-q)-\alpha B(X-r)=\alpha(-p+rq)<0,\]
when $\alpha$ has the opposite sign as $-p+rq$.
Next assume that $P(XB)=\infty$, as the $P(XB)=-\infty$
case is similar. First, we show 
that all gambles of the form $\alpha B(X-c)$ with $c$ real and
$\alpha\geq0$ are acceptable.  For each $\alpha\geq0$ and real $c$,
let $c_1=cq$ and $\beta=-\alpha c$.  Then
\[\alpha B(X-c)=\alpha(BX-c_1)+\beta(B-q),\]
and the sum of gambles on the right is acceptable.  To see that
$P(X|B)<\infty$ is incoherent$_1$, let $\alpha<0$, and let $c$
be real.  If $\alpha B(X-c)$ were acceptable, then the following
gamble would be acceptable for all real $d$:
\[\alpha B(X-c)-\alpha(XB-d)+\alpha c(B-q)=-\alpha(d+cq).\]
Let $d>-cq$ to see that the infimum is negative.
\end{proof} 
When $P(B)=0$, there may be multiple possible coherent$_1$
values of $P(X|B)$, but the set of possible values are determined from
$P(XB)$. 
\begin{lemma}\label{lem:pb0}
Let $\D=\{(X_i,B_i):i\in I\}$ be a collection of random
variable/nonempty event pairs with coherent$_1$ conditional previsions
$\{P(X_i|B_i):i\in I\}$.  Let $(X,B)$ be a random variable/nonempty
event pair such that $(B,\Omega)$ and $(XB,\Omega)$ are in $\D$ and
$P(B)=0$. 
\begin{enumerate} 
\item If $P(XB)\ne0$, the only possible
  coherent$_1$ value for $P(X|B)$ is the infinite value with the same
  sign as $P(XB)$.
\item If $P(XB)=0$, one can choose $P(X|B)$
  to be any extended-real number \label{part:pb02}
  $c_X$.
\end{enumerate} 
\end{lemma} 
\begin{proof}
\begin{enumerate}
\item  Let $c'$ have the same sign as $P(XB)$ if $P(XB)$
  is infinite, and let $c'=P(XB)$ otherwise.
To see that the specified infinite value is coherent$_1$,
  we will show that, for all real $c$ and all $\alpha$ with the
  same sign as $P(XB)$, $\alpha B(X-c)$ is an 
  acceptable gamble plus a positive number. For each
  such $\alpha$ and $c$,
\[\alpha B(X-c)=\alpha(BX-c')-\alpha cB+\alpha c',\]
where $\alpha c'>0$ and the other two gambles on the right are
acceptable.  To see that no other value is coherent$_1$,
suppose that we try to set $P(X|B)=p$, where $p$ is not the infinite
value with the same sign as $P(XB)$.  Then the following gamble
is acceptable, where $\alpha$ has the same sign as $P(XB)$,
and $c=p$ if $p$ is finite:
\[-\alpha B(X-c)+\alpha(XB-c')-\alpha cB=-c'\alpha<0.\]
Hence $P(X|B)=p$ is incoherent$_1$.
\item We show that all gambles of the form $\alpha_XB(X-c_X)$
  with $\alpha_X$ and $c_X$ real are acceptable.  For each
  such $\alpha_X$ and $c_X$, let $\beta=-\alpha_Xc_X$.  Then
\[\alpha_XB(X-c_X)=\alpha_XXB+\beta B,\]
and the two gambles on the right are acceptable.
\end{enumerate} 
\end{proof} 

We are now in a position to define finitely additive conditional
expectation.
\begin{definition}
{\rm Let $\L$ be a linear space of random variables, and let $\B$ be a
  collection of nonempty events that includes $\Omega$.  Suppose that
  $XB\in\L$ for every $X\in\L$ and $B\in\B$. Let $L$ be
a finitely additive expectation on $\L$.  Let
$L(\cdot|\cdot):\L\times\B\rightarrow\Real\cup\{\pm\infty\}$ be such
that for each $B\in\B$, $L(\cdot|B)$ is a finitely additive
expectation on $\L$ with $L(XB|B)=L(X|B)$ for all $X\in\L$ and such
that $\{L(X|B):X\in\L,B\in\B\}$ is a coherent$_1$ conditional
prevision.  Then we call $L(\cdot|\cdot)$ a {\em finitely additive
  conditional expectation}.}
\end{definition} 

The following result, whose proof is trivial, says that coherent$_1$
conditional prevision given an event of positive probability defines a
finitely additive conditional expectation.
\begin{proposition}\label{pro:gt0}
Let $P$ be a coherent$_1$ prevision on a linear space $\L$ containing
all constants.  Let $\B=\{B:P(B)>0\}$.  Assume that $XB\in\L$ for
each $X\in\L$ and $B\in\B$. Define $L(X|B)=P(XB)/P(B)$  
for all $X\in\L$ and $B\in\B$.  Then $L(\cdot|\cdot)$ is a finitely
additive conditional expectation. 
\end{proposition} 
In light of Lemma~\ref{lem:pb0}, it is clear that not every
coherent$_1$ conditional prevision given an event of 0 probability
defines a finitely additive conditional expectation.  After we prove
the Fundamental Theorem of Prevision, we can show
(Lemma~\ref{lem:ftcp}) that, if we start with a coherent$_1$ marginal
prevision, then for each 
nonempty event with 0 probability, there exists a coherent$_1$
conditional prevision that defines a a finitely additive conditional
expectation.

\section{The Fundamental Theorem of Prevision}\label{sec:ftp}
Attempts to define a unique expectation from a probability often fail
to provide an expectation for a random variable of the form $X-Y$
where both $X$ and $Y$ have infinite integral, especially when $XY$ is
identically 0.  Finitely additive expectations for such random
variables are guaranteed to exist, but they may not be unique.  The
tool for extending a coherent$_1$ collection of (conditional)
previsions to a larger collection is the fundamental theorem of prevision.

De Finetti (1974) proved an elementary version of the fundamental
theorem of prevision. That versions said that, if a collection of
events has been assigned coherent$_1$ previsions, then for each
additional event $E$, there is a nonempty closed interval such that
one can coherently choose $P(E)$ to equal any number in that interval.
Here, we prove a very general version that 
applies to unbounded random variables, infinite previsions,
conditional previsions, and extensions to arbitrary collections of
random variables.

The main step in the general version of the fundamental theorem is to
start with a collection of coherent$_1$ conditional previsions and add
one additional coherent$_1$ marginal prevision.
\begin{lemma}\label{lem:ftp}
Let $\D=\{(X_i,B_i):i\in I\}$ be a collection of random
variable/nonempty event pairs with coherent$_1$ conditional previsions
$\{P(X_i|B_i):i\in I\}$.  Let $X$ be a random variable. Then there
exists a nonempty interval $[a,b]$ such that $P(X|\Omega)=p$ is
coherent$_1$ with the previsions in $\D$ if and only if $p\in[a,b]$.
\end{lemma} 
\begin{proof}
Let $\C$ be the set of acceptable gambles, and define
\begin{eqnarray*}
A&=&\{f:\exists Y\in\C\mbox{\ with\ }Y+f\leq X\},\\  
B&=&\{f:\exists Y\in\C\mbox{\ with\ }-Y+f\geq X\},\\
a&=&\sup A,\\
b&=&\inf B.
\end{eqnarray*} 
The first step is to show that $a\leq b$.  If either
$A=\emptyset$ or $B=\emptyset$, the inequality is
trivial, so assume that neither set is empty.
We need to show that for all $f^a\in A$ and $f^b\in B$,
$f^a\leq f^b$.  For all $f^a\in A$ and $f^b\in B$, there exist
$Y^a,Y^b\in\C$ be such that 
\[Y^a+f^a\leq X\leq -Y^b+f^b.\]
Hence
\[Y^a+Y^b\leq f^b-f^a.\]
Since $Y^a+Y^b\in\C$, $\sup_\omega[Y^a(\omega)+Y^b(\omega)]\geq0$, it
follows that $f^b-f^a\geq0$, which completes the first step.

The second step is to show that every value in the interval $[a,b]$ is
a coherent$_1$ value for $P(X|\Omega)$. Let $p\in[a,b]$.  We need to
show that, for every $Y\in\C$,
$\sup_\omega\{Y(\omega)+\alpha[X(\omega)-c]\}\geq0$, where $c=p$ if
$p$ is finite, and $\alpha$ has the same sign as $p$ if $p$ is infinite.
Consider first, the case when $p$ is finite.  Suppose, to the
contrary, that there exists $Y\in\C$ and $\alpha$ such that
$\sup_\omega\{Y(\omega)+\alpha[X(\omega)-p]\}=-\epsilon<0$.  Clearly
$\alpha\ne0$.  If $\alpha>0$, then $X\leq
p-(\epsilon/\alpha)-Y/\alpha$.  Since $\alpha>0$, $Y/\alpha\in\C$, so
$p-(\epsilon/\alpha)\in B$ and $p-(\epsilon/\alpha)\geq b$, which
contradicts $p\leq b$.  Similarly, if $\alpha<0$, then $X\geq
p-(\epsilon/\alpha)-Y/\alpha$. Since $\alpha<0$, $-Y/\alpha\in\C$, so
$p-(\epsilon/\alpha)\in A$ and $p-(\epsilon/alpha)\leq a$, which
contradicts $p\geq a$.  Next, consider the case in which $p=-\infty$
so that $\alpha<0$, $a=-\infty$, and $A=\emptyset$.  Then
$-Y/\alpha\in\C$, and $X\geq c-(\epsilon/\alpha)-Y/\alpha$, which means
that $c-(\epsilon/\alpha)\in A$, which contradicts $A=\emptyset$.
Finally, if $p=\infty$, then $\alpha>0$, $b=\infty$, $B=\emptyset$,
$Y/\alpha\in\C$,  $X\leq c-(\epsilon/\alpha)-Y/\alpha$,
$c-(\epsilon/\alpha)\in B$, a contradiction.

The last step is to prove that every number outside of the interval
$[a,b]$ is an incoherent$_1$ choice for $P(X|\Omega)$.  First, assume
that we choose $P(X|\Omega)=p<a$. Of  
necessity, $A\not=\emptyset$, $a>-\infty$, and for every $f\in A$, there
exists $Y\in\C$ such that $Y+f\leq X$.  Let $\alpha<0$,
let $c=p$ if $p$ is finite and $c<a$ if $p=-\infty$.  Next, choose
$f\in(c,a]$ and $Y\in\C$ such that $Y+f\leq X$.  Then $-\alpha Y\in\C$,
and 
\[\alpha(X-c)-\alpha Y\leq\alpha(f-c)<0,\]
showing that $P(X|\Omega)=p$ is incoherent$_1$.
Finally, assume that we choose $P(X|\Omega)=p>b$.  Of necessity,
$B\not=\emptyset$, $b<\infty$, and for every $f\in B$, there 
exists $Y\in\C$ such that $-Y+f\geq X$.  Let $\alpha>0$,
let $c=p$ if $p$ is finite and $c>b$ if $p=\infty$.  Next, choose
$f\in[b,c)$ and $Y\in\C$ such that $-Y+f\geq X$.  Then $\alpha Y\in\C$,
and 
\[\alpha(X-c)+\alpha Y\leq\alpha(f-c)<0,\]
showing that $P(X|\Omega)=p$ is incoherent$_1$.
\end{proof} 

As an example of Lemma~\ref{lem:ftp}, we return to
Example~\ref{exa:stpete}
\begin{example}
{\rm In Example~\ref{exa:stpete}, $\Omega$ is the set of integers. We
introduced a random variable 
$X=\sum_{k=1}^\infty k\{\omega=k\}$ after assigning previsions
$P(\{\omega=k\})=2^{-k}$ for $k\geq1$.  We then proved (among other
things) that every number in the interval $[2,\infty]$ could be chosen
as a coherent$_1$ prevision for $X$. This fact actually follows from
Lemma~\ref{lem:ftp}.  Let $\D=\{(\{\omega=k\},\Omega): k\mbox{\ an
  integer}\}$.  The acceptable gambles have the form
\[Y(\omega)=\sum_{k=1}^\infty\alpha_k\left[\{\omega=k\}-2^{-k}\right],\]
where only finitely many $\alpha_k$ are nonzero.  In order for $Y+f\leq X$ it is
necessary and sufficient that 
\[f\leq\sum_{k=1}^\infty\alpha_k2^{-k}+\inf_{k\geq1}(k-\alpha_k)
=2-\sum_{k=1}^\infty(k-\alpha_k)2^{-k}+\inf_{k\geq1}(k-\alpha_k).\]
Since $\sum_{k=1}^\infty(k-\alpha_k)2^{-k}\geq\inf_{k\geq1}(k-\alpha_k)$, we have
$f\leq2$.  We can get $f$ as close as we want to 2 by choosing
$\alpha_k=k$ for all $1\leq k\leq n$ and $n$ large.  This makes $a=2$
in Lemma~\ref{lem:ftp}.  Since $\sup_\omega[-Y(\omega)+f]$ is finite
for all acceptable gambles and all $f$, $-Y+f\geq X$ is impossible.  Hence, the set
$B$ in the proof of Lemma~\ref{lem:ftp} is empty, and $b=\infty$.
}\end{example} 

We are now in position to prove the general extension theorem.

\begin{theorem}\label{thm:ftp}
Let $\D=\{(X_i,B_i):i\in I\}$ be a collection of random
variable/nonempty event pairs with coherent$_1$ conditional previsions
$\{P(X_i|B_i):i\in I\}$ and that
contains all pairs of the form $(c,\Omega)$ with $c$ a real number.
Let $\F$ be a set of random variable/nonempty event pairs that
contains $\D$. Then there exists a coherent$_1$ extension $P'$ of $P$ to $\F$.
\end{theorem}
\begin{proof}
We prove the result by extending $P$ one gamble at a time and applying
transfinite induction to cover the whole set $\F$. Let $\D_0=\D$ and
$P_0=P$. Well-order the set $\F\setminus\D$ as
$\{(X_\gamma,B_\gamma):1\leq\gamma\leq\Gamma\}$.  For each
$\gamma\leq\Gamma$, let 
$\D_\gamma=\D\cup\{(X_\beta,B_\beta):1\leq\beta\leq\gamma\}$.  For each
successor ordinal $\gamma+1$, assume that we have a coherent$_1$ extension
$P_\gamma$ to $\D_\gamma$.  (The assumption is true by hypothesis when
$\gamma=0$.) We extend $P_\gamma$ to $\D_{\gamma+1}$ as follows.
Apply Lemma~\ref{lem:ftp} with $\D=\D_\gamma$ and $X=B_{\gamma+1}$ to find a
coherent$_1$ marginal prevision for $B_{\gamma+1}$.  Let
$\D'=\D_\gamma\cup\{(B_{\gamma+1},\Omega)\}$.  Apply Lemma~\ref{lem:ftp}
with $\D=\D'$ and $X=X_{\gamma+1} B_{\gamma+1}$ to find a coherent$_1$
marginal prevision for $X_{\gamma+1} B_{\gamma+1}$.  Apply Lemma~\ref{lem:gt0}
or Lemma~\ref{lem:pb0} to find a coherent$_1$ conditional prevision
for $X_{\gamma+1}$ given $B_{\gamma+1}$.  Set
$P_{\gamma+1}(X_{\gamma+1}|B_{\gamma+1})$ equal to this coherent$_1$
conditional prevision.  This completes the induction step for
successor ordinals.

If $\gamma$ is a limit ordinal, assume that we have a coherent$_1$
extension of $P$ to $P_\beta$ for all $\beta<\gamma$.  Each
$(X,B)\in\D_\gamma$ is either in $\D$ or equals $(X_\beta,B_\beta)$ for some
$\beta<\gamma$.  So we define $P_\gamma(X|B)=P_\beta(X|B)$. 
These previsions are coherent$_1$ because every finite collection
appears in the induction at some $\beta<\gamma$.
This completes the proof of the induction step.
\end{proof} 

Finally, we can prove the existence of finitely additive conditional
expectations that agree with coherent$_1$ conditional previsions.
\begin{lemma}\label{lem:ftcp}
Let $\D$ be a collection of random variables that contains all
constants with coherent$_1$ marginal
previsions $P(\cdot)$. Let $\L$ be the linear span of all functions of the form
$XB$ for $X\in\D$ and $B$ a nonempty event. There exists a
coherent$_1$ conditional prevision on the set $\{(X,B):X\in\L,
B\ne\emptyset\}$ such that $P(\cdot|B)$ is a finitely additive
conditional expectation for each nonempty $B$.
\end{lemma} 
\begin{proof}
Use Theorem~\ref{thm:ftp} to extend $P$ to $\{(X,\Omega): X\in\L\}$,
which includes all pairs of the form $(B,\Omega)$ with
$B\ne\emptyset$.  Lemma~\ref{lem:gt0} and the first part of
Lemma~\ref{lem:pb0} specify all of the conditional previsions whose
coherent$_1$ values are uniquely determined from the marginal
previsions.  The proof of the second part of Lemma~\ref{lem:pb0}
actually shows that $P(X|B)$ can simultaneously be set to arbitrary
values for all $(X,B)$ with $P(B)=P(XB)=0$.  The proof will be
complete if we can choose values for all such $P(X|B)$ so that the
resulting $P(\cdot|\cdot)$ is a finitely additive conditional
expectation.  The only restrictions that we have to obey are those
caused by conditional previsions fixed by the first part of
Lemma~\ref{lem:pb0}.  For each $B$ with $P(B)=0$ and each $X\in\L$,
let $X_B:B\rightarrow\Real$ be defined by $X_B(\omega)=X(\omega)$ for
$\omega\in B$. That is, $X_B$ is $X$ restricted to domain $B$. Let
$\D_B$ be the set of all $(X_B,B)$ such that $X\in\L$ and either 
$P(X|B)$ is uniquely determined by the first part of
Lemma~\ref{lem:pb0} or $X$ is constant on $B$.  For constant $X$, set
$Q(X_B|B)$ equal to that constant, and for all other $X$, let
$Q(X_B|B)=P(X|B)$.  This makes $Q(\cdot|B)$ a coherent$_1$ marginal
prevision on $\D_B$ with $B$ as the state space.  Use Theorem~\ref{thm:ftp} with
$\Omega=B$ and $\D=\D_B$ to extend $Q$ to $\{(X_B,B): X\in\L\}$, which
makes $Q(\cdot|B)$ a finitely additive expectation on $\{X_B:X\in\L\}$.
Define $P(X|B)=Q(X_B|B)$, and note that $P(XB|B)=P(X|B)$ and
$P(\cdot|\cdot)$ is a finitely additive conditional expectation.
\end{proof}

\section{Discussion}
In this paper we extend both of de Finetti's (1974) concepts of
coherent prevision to unbounded random variables and infinite
previsions. We define infinite prevision so that it makes sense in the
gambling interpretation of prevision.  We define how proper scoring
rules can be used to score potentially infinite previsions.  We extend
the equivalence of the avoidance of sure loss from gambling and the
nonexistence of uniformly smaller scores to a large class of strictly
proper scoring rules and potentially infinite previsions.  We use a
finitely additive extension of the concept of Daniell integral to
define finitely additive expectation on an arbitrary linear space of
random variables, and show that it is equivalent to our extension of
coherence to include infinite previsions.  We give a version of the
fundamental theorem of prevision that applies to unbounded random
variables, infinite previsions, and conditional previsions.  The
fundamental theorem allows extension of a coherent conditional
prevision from an arbitrary set of random variables to an arbitrary
larger set of random variables, including the set of all random variables.


\begin{thebibliography}{8}
\expandafter\ifx\csname natexlab\endcsname\relax\def\natexlab#1{#1}\fi
\expandafter\ifx\csname url\endcsname\relax
  \def\url#1{\texttt{#1}}\fi
\expandafter\ifx\csname urlprefix\endcsname\relax\def\urlprefix{URL }\fi
\providecommand{\eprint}[2][]{\url{#2}}

\comment{
\bibitem[{Berti-et~al. (2001)}]{berti-etal2001}
{\sc Berti, P.}, {\sc Regazzini, E.} and {\sc Rigo, P.} (2001).
\newblock Strong previsions of random elements. {\em Statistical Methods and
  Applications} {\bf 10} 11--28.


\bibitem[{Berti and Rigo (2000)}]{berti-rigo2000}
{\sc Berti, P.} and {\sc Rigo, P.} (2000).
\newblock Integral representation of linear functionals on spaces of unbounded
  functions. {\em Proceedings of the American Mathematical Society} {\bf 128}
  3251--3258.

\bibitem[{Berti and Rigo (2002)}]{berti-rigo2002}
{\sc Berti, P.} and {\sc Rigo, P.} (2002).
\newblock On coherent conditional probabilities and disintegrations. {\em
  Annals of Mathematics and Artificial Intelligence} {\bf 35} 71--82.
}
\bibitem[{Crisma, Gigante, and Millossovich (1997)}]{crisma97}
{\sc Crisma, L.}, {\sc Gigante, P.}, and {\sc Millossovich, P.} (1997).
\newblock A notion of coherent prevision for arbitrary random quantities.
{\em Journal of the Italian Statistical Society} {\bf 6} 233--243.

\bibitem[{Crisma and Gigante (2001)}]{crisma01}
{\sc Crisma, L.} and {\sc Gigante, P.} (2001).
\newblock A notion of coherent conditional prevision for arbitrary random quantities.
{\em Statistical Methods \& Applications} {\bf 10} 29--40.

\comment{
\bibitem[{de~Finetti(1949)}]{deFinetti1949}
{\sc de~Finetti, B.} (1949).
\newblock {\em Probability, Induction, and Statistics},  chap. 5, On the
  Axiomatization of Probability.
\newblock John Wiley, New York.
}
\bibitem[{de~Finetti(1974)}]{deFinetti1974}
{\sc de~Finetti, B.} (1974).
\newblock {\em Theory of Probability},
\newblock Vol.~1, John Wiley, New York.

\comment{
\bibitem[{de~Finetti(1981)}]{deFinetti1981}
{\sc de~Finetti, B.} (1981).
\newblock The role of Dutch books and proper scoring rules. {\em Brit. J. Phil.
  Sci} {\bf 32} 55--56.
}
\bibitem[{dubins(1975)}]{dubins1975}
{\sc Dubins, L.} (1975).
\newblock Finitely additive conditional probabilities, conglomerability and
  disintegrations. {\em Annals of Probability} {\bf 3} 89--99.

\bibitem[{dunford and Schwartz(1958)}]{dunford-schwartz1958}
{\sc Dunford, N.} and {\sc Schwartz, J.} (1958).
\newblock {\em Linear Operators\/} (Part 1 ed.) John Wiley, New York.

\bibitem[{Gneiting (2011a)}]{gneiting11a}
{\sc Gneiting, T.} (2011a). 
\newblock Making and Evaluating Point Forecasts. {\em Journal of the
  American Statistical Association} {\bf 106} 746--762.

\bibitem[{Gneiting (2011b)}]{gneiting11b}
{\sc Gneiting, T.} (2011b). 
\newblock Quantiles as Optimal Point Forecasts. {\em International Journal of
  Forecasting} {\bf 27} 197--207. 

\comment{
\bibitem[{levi(1980)}]{Levi1980}
{\sc Levi, I.} (1980).
\newblock {\em The Enterprise of Knowledge}, MIT Press, Cambridge, MA.}

\bibitem[{regazzini(1987)}]{regazzini1987}
{\sc Regazzini, E.} (1987).
\newblock De finetti's coherence and statistical inference. {\em Annals of
  Statistics} {\bf 15} 845--864.

\comment{
\bibitem[{royden(1968)}]{royden1968}
{\sc Royden, H.} (1968).
\newblock {\em Real Analysis (Second Edition)}, Macmillan, New York.
}
\bibitem[{Savage (1971)}]{savage71}
{\sc Savage, L. J.} (1971).
\newblock Elicitation of Personal Probabilities and Expectations. {\em
  Journal of the American Statistical Association} {\bf 66}
  783--801.

\comment{
\bibitem[{Schervish et~al(1984)}]{Schervish,Seidenfeld and Kadane1984}
{\sc Schervish, M.}, {\sc Seidenfeld, T.} and {\sc Kadane, J.} (1984).
\newblock On the extent of non-conglomerability of finitely additive
  probabilities. {\em Z.f. Wahr.} {\bf 66} 205--226.

\bibitem[{Schervish et~al(2008a)}]{ssk2008a}
{\sc Schervish, M.}, {\sc Seidenfeld, T.} and {\sc Kadane, J.} (2008a).
\newblock Fundamental theorems of prevision and asset pricing. {\em
  International Journal of Approximate Reasoning} {\bf 49} 148--158.

\bibitem[{Schervish, Seidenfeld and Kadane(2008b)}]{ssk2008b}
{\sc Schervish, M.}, {\sc Seidenfeld, T.} and {\sc Kadane, J.} (2008b).
\newblock On the equivalence of conglomerability and disintegrability for
  unbounded random variables. Technical Report 864, Carnegie Mellon University.

\bibitem[{Schervish, Seidenfeld and Kadane(2009b)}]{ssk2009b}
{\sc Schervish, M.}, {\sc Seidenfeld, T.} and {\sc Kadane, J.} (2009).
\newblock Proper scoring rules, dominated forecasts, and coherence. {\em
  Decision Analysis} 6(4) 202--221.
}
\bibitem[{Schervish, Seidenfeld and Kadane(2013)}]{ssk2013}
{\sc Schervish, M.}, {\sc Seidenfeld, T.} and {\sc Kadane, J.} (2013).
\newblock Dominating countably many forecasts. Submitted.

\comment{
\bibitem[{Seidenfeld, Schervish and Kadane(2009a)}]{ssk2009a}
{\sc Seidenfeld, T.}, {\sc Schervish, M.} and {\sc Kadane, J.} (2009).
\newblock Preference for equivalent random variables: {A} price for unbounded
  random variables. {\em Journal of Mathematical Economics} {\bf 45} 329--340.
}

\bibitem[{Williams(1975)}]{williams1975}
{\sc Williams, P. M.} (2007).
\newblock Notes on conditional previsions. {\em
  International Journal of Approximate Reasoning} {\bf 44} 366--383.
\end{thebibliography}
\end{document}